\documentclass[a4paper,12pt]{amsart}
\usepackage[utf8]{inputenc}
\usepackage[T2A]{fontenc}
\usepackage[english]{babel}

\usepackage{amssymb}
\usepackage{amsthm}

\newtheorem{theorem}{Theorem}

\def\hm#1{#1\nobreak\discretionary{}{\hbox{\ensuremath{#1}}}{}}

\def\Z{{\mathbb Z}}
\def\N{{\mathbb N}}

\def\taug{{\mathfrak t}}
\def\p{{\mathfrak p}}

\def\eps{\varepsilon}

\def\res{\mathop{\mathrm{res}}}

\def\le{\leqslant}
\def\ge{\geqslant}

\usepackage[unicode]{hyperref}
\usepackage{breakurl}
\usepackage{array}

\begin{document}

\title[Dirichlet divisor problem on Gaussian integers]{Dirichlet divisor problem \\ on Gaussian integers}
\author{Andrew V. Lelechenko}
\address{I.~I.~Mechnikov Odessa National University}
\email{1@dxdy.ru}



\begin{abstract}
We improve existing estimates of moments of the Riemann zeta function.
As a consequence, we are able to derive new estimates for the asymptotic behaviour of $\sum_{N \alpha \le x} \taug_k(\alpha)$, where~$N$ stands for the norm of a complex number and $\taug_k$ is the $k$-di\-men\-sion\-al divisor function on Gaussian integers.
\end{abstract}

\maketitle

\section{Introduction}

Define the divisor function $\tau$ and its generalisation, the $k$-dimensional divisor function $\tau_k$ as follows:
$$
\tau(n) = \sum_{d \mid n \atop d > 0} 1,
\qquad
\tau_k(n) = \sum_{d_1\cdots d_k = n \atop d_1,\ldots,d_k > 0} 1,
$$
where $n$ is a non-zero integer.
That said, $\tau \equiv \tau_2$.

One can study asymptotic properties of summatory functions
$$
\sum_{0<n \le x} \tau(n) \qquad \text{and} \qquad \sum_{0<n \le x} \tau_k(n).
$$
It can be shown (see \cite[(6.27)]{kratzel1988}) that for growing $x \to \infty$ we have
\begin{align*}
\sum_{0<n \le x} \tau(n) &= x \ln x + (2 \gamma - 1) x + O(x^{\theta_2}),
\\
\sum_{0<n \le x} \tau_k(n) &= x P_k(\ln x) + O(x^{\theta_k}),
\end{align*}
where $\gamma \approx 0.577$ denotes the Euler–Mascheroni constant, $P_k(y)$ is a fixed univariate polynomial of degree $k-1$, and $\theta_k$ is a positive real. The quest of deriving precise estimates for $\theta_k$ is one of the central problems of multiplicative number theory, known as Dirichlet divisor problem. It has been widely studied by various authors: Dirichlet has proved that $\theta_2 \le 1/2$, Vorono{\"\i} improved this estimate to $\theta_2 \le 1/3+\eps$ \cite{voronoi1903,voronoi1904}, and more modern developments can be found in \cite{ivic1989,kolesnik1981,titchmarsh1986}. Namely,
$$
\theta_2 \le 131/416,
\qquad
\theta_3 \le 43/96,
\qquad
\theta_k \le {k-1 \over k+2}, \quad k \ge 4.
$$

It is natural to extend the notion of divisor functions from integers to other unique factorisation domains such as rings of quadratic integers. Let $R$ be a principal ring of imaginary quadratic integers. Then it is possible to define divisor functions $\taug, \taug_k \colon R \setminus {0} \to \N$ as
$$
\taug(\alpha) = \sum_{d \mid \alpha \atop d \in R/U(R)} 1,
\qquad
\taug_k(\alpha) = \sum_{d_1\cdots d_k \sim \alpha \atop d_1,\ldots,d_k \in R/U(R)} 1,
$$
where $\alpha \sim \beta$ iff $\alpha \mid \beta$ and $\beta \mid \alpha$, and $U(R)$ is a group of units of $R$.

Further, if $R$ is also a Euclidean ring equipped
with norm $N\colon R \hm\to \N$, one can learn asymptotic properties of
$\sum_{0<N \alpha \le x} \taug_k(\alpha)$.
Lai Dyk Thin~\cite{thin1965} has proved that
\begin{align}
\label{eq:thin1}
\sum_{0<N\alpha \le x} \taug (\alpha) &= l_1 x \ln x + l_2 x + O(x^{3/5+\eps}),
\\
\label{eq:thin2}
\sum_{0<N\alpha \le x} \taug_k (\alpha) &= x L_k(\ln x) + O(x^{1-1/(k+1)+\eps}),
\end{align}
where constants $l_1$, $l_2$ and polynomial $L_k$ of degree $k-1$ depend only on $R$.

The aim of following notes is to improve the error term in \eqref{eq:thin1} and~\eqref{eq:thin2} in the specific case of Gaussian integers $\Z[i]$. Namely, from now on
$$
\taug(\alpha) = \sum_{d \mid \alpha \atop \arg d \in [0,\pi/2)} 1,
\qquad
\taug_k(\alpha) = \sum_{d_1\cdots d_k \sim \alpha \atop \arg d_1,\ldots,\arg d_k \in [0,\pi/2)} 1.
$$
and $N(a+bi) = a^2 + b^2$. In order to obtain this result we improve known estimates
of moments of the Riemann zeta function.

\section{Moments of $\zeta$}

As usual $\zeta(s)$ is the Riemann zeta function.
Real and imaginary com\-po\-nents of the complex~$s$ are denoted as $\sigma:=\Re s$ and~$t:=\Im s$, so~$s=\sigma+it$.

Denote by $M(A)$ a real function such that
\begin{equation}\label{eq:M-def}
\int_1^T |\zeta(1/2+it)|^A \, dt \ll T^{M(A)+\eps}.
\end{equation}

Estimates of the Riemann zeta function on critical line $\sigma = 1/2$
are crucial for many applications. The best known result can be found
in~\cite[Th. 8.3]{ivic2003}. The following theorem improves it for $A>12$.

\begin{theorem}\label{th:M(A)}
The following choice of $M$ is valid and satisfies \eqref{eq:M-def}:
\begin{equation}
\label{eq:M-estimate}
M(A) = \begin{cases}
1 + (A - 4) / 8, & 4\le A\le12, \\
1 + \max\bigl\{ {13\over 84} (A-6), E(A) \bigr\}, & 12 < A\le C, \\
1 + {13\over 84} (A-6), & C< A,
\end{cases}
\end{equation}
where
\begin{align}
\label{eq:M(A)-rat-obj}
E(A) &= \inf_{(k,l) \in P} \left\{ {l\over k} \biggm| (4-A)k+4l+2\ge 0 \right\},
\\
\label{eq:C}
C &= \inf_{(k,l) \in P} \left\{ {4k+4l+2 \over k} \biggm| 1-k-{16\over13} l \ge 0 \right\}
= {16645467 \over 972266} = 17.1\ldots
\end{align}
and $P$ denotes the set of exponent pairs,
defined in accordance to Krät\-zel \cite[Ch. 2]{kratzel1988}.
\end{theorem}
\begin{proof}
The first case follows from the estimates
$\int_1^T |\zeta(1/2+it)|^4\, dt \hm\ll T \log^4 T$ by Ingham~\cite{ingham1926} and $\int_1^T |\zeta(1/2+it)|^{12}\, dt \ll T^2 \log^{17} T$ by Heath-Brown~\cite{heathbrown1978}.
Precisely, let us denote $\xi(t) = |\zeta(1/2+it)|$. By Hölder inequality we have
\begin{multline*}
\int_1^T \xi^A(t) \, dt
=
\int_1^T \xi^{(12-A)/2}(t) \xi^{(3A-12)/2}(t)\, dt
\ll
\\
\ll
\left(\int_1^T \xi^{q_1(12-A)/2}(t)\,dt\right)^{1/q_1}
\left(\int_1^T \xi^{q_2(3A-12)/2}(t)\, dt\right)^{1/q_2}
\end{multline*}
for $1/q_1+1/q_2=1$. Taking $q_1=8/(12-A)$ and $q_2=8/(A-4)$ we get
\begin{multline*}
\int_1^T \zeta^A (1/2+it)\, dt
\ll \\ \ll
\left(\int_1^T \xi^4 (t)\, dt\right)^{(12-A)/8}
\left(\int_1^T \xi^{12} (t)\, dt\right)^{(A-4)/8}
\ll
\\
\ll
T^{(12-A)/8} \log^{4(12-A)/8} T
\cdot
T^{2(A-4)/8} \log^{17(A-4)/8} T
=\\=
T^{1+(A-4)/8} \log^{(13A-20)/8} T.
\end{multline*}

\medskip

Consider the second case. Denoting $R$ and $V$ as in Ivić~\cite[(8.6)]{ivic2003}. It is enough to show that for $A>12$ and for every exponent pair $(k,l) \in P$ such that $(4-A)k+4l+2\ge0$ we have
\begin{equation}\label{eq:second-case}
R \ll T^{1+\max\{{13 \over 84}(A-6), l/k\} + \eps} V^{-A}.
\end{equation}
But by \cite[Th.~8.2]{ivic2003} we have
$$
R
\ll T^{1+\eps} V^{-6}  + T^{1+l/k+\eps} V^{-2(1+2k+2l)/k}.
$$
By condition on $(k,l)$ we know that $V^{-2(1+2k+2l)/k} \ll V^{-A}$, so
$$
R \ll T^{1+\eps} V^{-A} (V^{A-6} + T^{l/k}).
$$
But definitely since $\zeta(1/2+it)\ll t^{13/84+\eps}$ by Bourgain~\cite{bourgain2017} we have~$V \hm\ll T^{13/84+\eps}$, which completes the proof of~\eqref{eq:second-case}.

\medskip

Now let us investigate the third case. Let $A>C$. Then by definition of $C$ there is an exponent pair $(k,l)$ such that $1-k-16l/13 \ge 0$ and~$(4k+4l+2)/k < A$. So by \cite[Th.~8.2]{ivic2003} we have
\begin{multline*}
R
\ll T^{1+\eps} V^{-6} (1 + T^{l/k} V^{-(-2k+4l+2)/k})
\ll \\
\ll \begin{cases}
T^{1+\eps} V^{-6}, & V \gg T^c,\\
T^{1+l/k+\eps} V^{-(4k+4l+2)/k}, & \text{otherwise},
\end{cases}
\end{multline*}
where $c = l / (-2k+4l+2)$. One can check that condition $1-k\hm-16l/13 \ge 0$ implies $c\le 13/84$.

It is enough to prove that
$$
S := \sum_{r\le R} |\zeta (1/2+it_r)|^A \ll T^{1 + {13 \over 84} (A-6) + \eps},
$$
where $|t_r| \le T$, $|t_r-t_s| \ge 1$ for $1\le r \ne s \le R$.
Cf. \cite[(8.58)]{ivic2003}.

Again since $\zeta(1/2+it)\ll t^{13/84+\eps}$ we can split $\{t_r\}_{r=1}^R$ into $1\hm+\lfloor{13\over84} \log T \rfloor$ disjoint subsets $\{t_{j,r_j}\}_{r_j=1}^{R_j}$ such that for every $t_{j,r_j}$ we have~$V_j \hm{:=} 2^j \le |\zeta(1/2+it_{j,r_j})| \le 2^{j+1}$. Then
$$
S \ll \sum_{j=0}^{\lfloor{13\over84} \log T \rfloor} \sum_{r\le R_j} 2^{Aj}
\ll \left(\sum_{j=0}^{c\log T} + \sum_{j=c\log T}^{\lfloor{13\over84} \log T \rfloor}\right) R_j 2^{Aj}  =: S_2 + S_1.
$$
Here by choice of $c$ for $j\ge c\log T$ we have $R_j \ll T^{1+\eps} 2^{-6j}$, so
$$
S_1 \ll T^{1+\eps} \sum_{j=c\log T}^{\lfloor{13\over84} \log T \rfloor} 2^{(A-6)j} \ll T^{1+{13\over84} (A-6) + \eps}.
$$
On the other side for $j\le c\log T$ we have $R_j \ll T^{1+l/k+\eps} 2^{-(4k+4l+2)j/k}$, so
$$
S_2 \ll  T^{1+l/k+\eps} \sum_{j=0}^{c\log T} 2^{(A-(4k+4l+2)/k)j}
\ll T^{1+l/k+(A-(4k+4l+2)/k)c+\eps}.
$$
But
\begin{multline*}
l/k+(A-(4k+4l+2)/k)c
= l/k+(A-6-(-2k+4l+2)/k)c
=\\= l/k+(A-6)c-l/k
= (A-6)c \le {13\over84}(A-6),
\end{multline*}
which completes the proof.
\end{proof}

To apply Theorem~\ref{th:M(A)} efficiently, we need a method to evaluate infimums of form \eqref{eq:M(A)-rat-obj} and \eqref{eq:C} over exponent pairs. We have developed such framework, whose initial version has been described in \cite{lelechenko2013-acta}. Since then the framework has been developed further and
released as a {\tt exp-pairs} package \cite{exp-pairs}. For instance, an estimate of $C$,
given in \eqref{eq:C}, corresponds to the choice
$$
(k, l) = BA (A (BA A)^2)^2 A^2 (BA)^2 A BA (13/84, 55/84),
$$
where $A$ and $B$ stand for application of $A$- and $B$-process from Krät\-zel~\cite[Ch. 2]{kratzel1988}, and $(13/84, 55/84)$ is an exponent pair by Bourgain~\cite{bourgain2017}.

We computed Table~\ref{t:M_13_14_15} as a reference for estimates of $M(A)$ provided by Theorem~\ref{th:M(A)} for integer $A \in (12,C)$.

\begin{table}
\center\begin{tabular}{cp{2.4cm}p{8.8cm}}
$A$ & $M(A)$ & Exponent pair for $E(A)$ \\
13\vphantom{$\Biggl($}
& 2.134766\ldots
& $\displaystyle{BAA^3(BA)^2A^2BAA(BA)^5A^5BA(A(BA)^2A)^2 \cdot{} \atop {}\cdot BA((BA)^2A)^2(ABA)^3BAA^4(BA)^7A^4BAA^5\cdots}$ \\
14\vphantom{$\Biggl($}
& $\displaystyle{1117297289 \over 491431296}$
& $\displaystyle{BAA^2BA(BAA)^5((BA)^2A)^2A(BA)^3A^3 \cdot{} \atop {}\cdot (ABA)^2 (1/6,2/3)}$ \\
15\vphantom{$\Biggl($}
& $\displaystyle{61902400787 \over 25629743097}$
& $\displaystyle{(BAA^2)^2(BA)^2A^4(BA)^4ABA(A(BA)^3A)^2 \cdot{} \atop {}\cdot (BA)^5A (1/6,2/3)}$
\\
16\vphantom{$\Biggl($}
& $\displaystyle{2.558254\ldots}$
& $\displaystyle{BA A \bigl((BA)^2 A BA A (BA)^2 A^2 (BA)^4\bigr)^\infty}$
\\
17\vphantom{$\Biggl($}
& $\displaystyle{7321 \over 2708}$
& $\displaystyle{(BA A)^3 A^2 (13/84, 55/84)}$
\end{tabular}
\par~
\caption{Values of $M(A)$ for $A=13,14,15,16,17$.}
\label{t:M_13_14_15}
\end{table}

\section{Moments of $Z$}

Let us briefly recap key properties of Gaussian integers.
The ring~$\Z[i]$ consists of $\alpha = a+bi$ for integer $a$ and $b$.

There are four units of the ring: $U(\Z[i]) = \{ 1, i, -1, -i \}$.
For any~$\alpha \hm\ne 0$ its orbit under action of the unit group consists
of four elements, one per each quadrant. We will use an element of the orbit
from the first quadrant (such that $\arg\alpha \in [0,\pi/2]$) as
a canonical representative.

The ring is equipped with norm $N(a+bi) = a^2+b^2$, which is
a homomorphism of multiplicative group: $N(\alpha \cdot \beta) = N(\alpha) \cdot N(\beta)$.
The ring is Euclidean and principal, so it is a unique factorisation domain.

Gaussian integer $\p$ is prime if and only if one of the following cases has place:
\begin{itemize}
\item $\p \sim 1+i$,
\item $\p \sim p$, where $p \equiv 3 \pmod 4$,
\item $N(\p) = p$, where $p \equiv 1 \pmod 4$.
\end{itemize}
In the last case there are exactly two non-associated $\p_1$ and~$\p_2$ such that $N(\p_1) \hm= N(\p_2) = p$.
See \cite[\S34]{gauss1832}.

Let $Z(s)$ denote the Dedekind zeta function of $\Z[i]$, which is a Gaussian analogue of the Riemann zeta function. Namely,
$$
Z(s) = \sum_{\alpha \ne 0 \atop \arg\alpha \in [0,\pi/2)} N(\alpha)^{-s}.
$$

Let $\beta$ be the Dirichlet beta function,
$$
\beta(s) = \sum_{n\ge0} {(-1)^n \over (2n+1)^{-s}}.
$$
Converting Dirichlet sums to Euler products and back, we have
\begin{align*}
\label{eq:hecke-product}
Z(s)
& = \prod_{\p \atop \arg\p\in[0,\pi/2)} (1-N(\p)^{-s})^{-1}
=\\ \nonumber
& = {1 \over 1-2^{-s} }
  \prod_{p \equiv 3 \!\!\!\! \pmod 4} {1 \over 1-p^{-2s}}
  \prod_{p \equiv 1 \!\!\!\! \pmod 4} {1 \over (1-p^{-s})^2}
=\\ \nonumber
& = \prod_p {1 \over 1-p^{-s}}
  \prod_{p \equiv 3 \!\!\!\! \pmod 4} {1 \over 1+p^{-s}}
  \prod_{p \equiv 1 \!\!\!\! \pmod 4} {1 \over 1-p^{-s}}
=\\ \nonumber
& = \zeta(s) \beta(s).
\end{align*}

Denote by $I(A)$ a real function such that
\begin{equation*}
\int_1^T |Z(1/2+it)|^A \, dt \ll T^{I(A)+\eps}.
\end{equation*}

Firstly, $I(2) = 1$, because
\begin{align*}
\int_1^T \!\!\!\! |Z(1/2+it)|^2 \, dt
&\ll \left( \int_1^T \!\!\!\! |\zeta(1/2+it)|^4 \, dt  \right)^{\!\! 1/2} \!\!
  \left( \int_1^T \!\!\!\! |\beta(1/2+it)|^4 \, dt  \right)^{\!\! 1/2}
\!\!\!\!\!\! \ll \\
&\ll T^{(1+\eps)/2} T^{(1/2+\eps)/2}
= T^{1+\eps},
\end{align*}
where we applied estimates for fourth moments from Montgomery \cite[Th. 10.1]{montgomery1971}.

Secondly, $I(3) = 5/4$, because
\begin{align*}
\int_1^T \!\!\!\! |Z(1/2+it)|^3 \, dt
&\ll \left( \int_1^T \!\!\!\! |\zeta(1/2+it)|^{12} \, dt  \right)^{\!\! 1/4} \!\!
  \left( \int_1^T \!\!\!\! |\beta(1/2+it)|^4 \, dt  \right)^{\!\! 3/4}
\!\!\!\!\!\! \ll \\
&\ll T^{(2+\eps)/4} T^{3(1+\eps)/4}
= T^{5/4+\eps},
\end{align*}
where the estimate for the twelfth moment is by Ingham~\cite{ingham1926}.

Further, estimates for higher moments of Dedekind zeta function are given
by the following theorems.

\begin{theorem}
\begin{equation}
\label{eq:I-estimate}
I(A) = \begin{cases}
({13\over84} + (M(D)-{8\over21})/D) A + {8\over21}
\\ \qquad \qquad \qquad \,\,\,\, = 0.2896 A + {8\over21},
& 4\le A\le D, \\
{13\over 84} A + M(A),
& D\le A.
\end{cases}
\end{equation}
where $D=12.4868\ldots$
\end{theorem}
\begin{proof}
For brevity below
$Z^A$ means $|Z(1/2\hm+it)|^A$,
$\zeta^A$ means $|\zeta(1/2\hm+it)|^A$,
$\beta^A$ means $|\beta(1/2\hm+it)|^A$,
and $\int\cdot$ stands for $\int_1^T \cdot \, dt$.

We start with the first case. Let $D$ be any real greater than $A$.
Then for $b = 4(D-A)/D$ we write
$$
\int Z^A
= \int \zeta^A \beta^A
= \int \zeta^A \beta^b \beta^{A-b}.
$$
Since $\beta \ll T^{13/84+\eps}$ for $t\in[1,T]$ we get
\begin{equation}
\label{eq:Z-moment-1}
\int Z^A
\ll T^{13(A-b)/84 + \eps}
    \int \zeta^A \beta^b.
\end{equation}
Now apply Hölder inequality with $q_1 = D/A$, $q_2 = D/(D-A)$, $1/q_1 \hm+ 1/q_2 = 1$
to obtain
\begin{equation}
\label{eq:Z-moment-2}
\int \zeta^A \beta^b
\ll \left(\int \zeta^D\right)^{A/D} \left(\int \beta^4\right)^{(D-A)/D}
\ll T^{A M(D) / D + \eps} T^{(D-A)/D + \eps}.
\end{equation}
Combination of \eqref{eq:Z-moment-1} and \eqref{eq:Z-moment-2}
provides us with the first statement of the theorem.

It remains to choose $D$ to minimize $(M(D)-{8 \over 21})/D$
and numerical computations by Theorem \ref{th:M(A)} give us $D=12.4868\ldots$

The second statement is almost trivial, since
$$
\int Z^A = \int \zeta^A \beta^A
\ll T^{13A/84+\eps} \int \zeta^A
\ll T^{13A/84+M(A)+\eps}.
$$
\end{proof}

\section{Summatory function of $\taug_k$}

Now we are ready to attack our main aim: the summatory function of $\taug_k$.
Just to get better acquainted we refer readers to Table \ref{t:taug}
for the plot of values of $\taug$.

\begin{table}
\center\begin{tabular}{r|rrrrrrrrrrrrrrr}
15&4&2&8&2&12&4&4&3&8&8&4&4&4&2&16\\
14&2&12&4&6&4&8&4&12&2&8&2&12&4&8&2\\
13&8&2&4&4&4&4&4&2&8&2&8&2&8&4&4\\
12&4&6&4&12&3&12&2&10&6&6&4&12&2&12&4\\
11&4&4&8&2&4&2&8&4&4&4&4&4&8&2&4\\
10&2&8&2&6&6&8&2&6&2&16&4&6&2&8&8\\
9&4&4&8&2&4&4&8&4&6&2&4&6&8&2&8\\
8&4&6&2&10&2&9&2&8&4&6&4&10&2&12&3\\
7&6&2&4&4&4&4&4&2&8&2&8&2&4&4&4\\
6&2&8&4&6&2&8&4&9&4&8&2&12&4&8&4\\
5&4&2&4&2&8&2&4&2&4&6&4&3&4&4&12\\
4&2&6&3&6&2&6&4&10&2&6&2&12&4&6&2\\
3&4&2&4&3&4&4&4&2&8&2&8&4&4&4&8\\
2&2&4&2&6&2&8&2&6&4&8&4&6&2&12&2\\
1&2&2&4&2&4&2&6&4&4&2&4&4&8&2&4\\
0&1&3&2&5&4&6&2&7&3&12&2&10&4&6&8\\\hline
&1&2&3&4&5&6&7&8&9&10&11&12&13&14&15
\end{tabular}
\par~
\caption{Values of $\taug(a+bi)$ in the first quadrant.
Here~$a$ increases horizontally and $b$ vertically. Numbers with exactly two divisors
are Gaussian primes.}
\label{t:taug}
\end{table}

Similar to the real case $\sum_{n>0} \tau_k(n) n^{-s} = \zeta^k(s)$,
one can check that
$$
\sum_{\alpha\ne0 \atop \arg \alpha \in [0,\pi/2)} \taug_k(\alpha) N(\alpha)^{-s}
= Z^k(s).
$$

\begin{theorem}\label{th:taug}
Let $S_k(x) = \sum_{0<N\alpha \le x \atop \arg \alpha \in [0,\pi/2)} \taug_k(\alpha)$.
Then
\begin{align}
\label{eq:taug-prediction}
S_2(x) &= l_1 x \ln x + l_2 x + O(x^{1/2+\eps}),
\\
\nonumber
S_3(x) &= x P_3(\log x) + O(x^{3/5+\eps}),
\\
\nonumber
S_k(x) &= x P_k(\log x) + O(x^{1-1/2I(k)+\eps}).
\end{align}
where
\begin{align*}
l_1 & = \pi^2 / 16 = 0.61685\ldots,
\\
l_2 & = \pi^2 (2\gamma - 1) / 16 + \pi \beta'(1) / 2 = 0.39827\ldots
\\
\deg P_k &= k - 1.
\end{align*}
\end{theorem}
\begin{proof}
By Perron formula we have
$$
S_k(x) = \res_{s=1} Z^k(s) x^s / s + O(xT^{-1} + x^{1/2} T^{I(k)-1 + \eps}).
$$
Here the residues give the main term of form $x P_k(\log x)$. Let us analyse the case
of $k=2$ in details. We have
\begin{align*}
\res_{s=1} Z^2(s) x^s / s
&= \res_{s=1} (\zeta^2(s)/s \cdot \beta^2(s) \cdot x^s)
= \\
&= \res_{s=1}
  \left( {1\over(s-1)^2} + {2\gamma-1\over s-1} + O(1) \right)
\times \\ &\times
  \left( {\pi\over4} + {\pi \beta'(1) \over2} (s-1) + O(s-1)^2 \right)^2
\times \\ &\times
  \left( x + x \log x (s-1) + O(s-1)^2 \right)
= \\ &= {\pi^2 \over 16} x \ln x
+ \left({{\pi^2 (2\gamma - 1) \over 16} + {\pi \beta'(1) \over 2}}\right) x.
\end{align*}

With regards to the error term let us choose $T=x^a$ in order to minimize
the magnitude of
$O(xT^{-1} + x^{1/2} T^{I(k)-1 + \eps})$. One can check
that $T=x^{1/2I(k)}$ turns this expression into $O(x^{1-1/2I(k)+\eps})$.

Error terms for $k=2$ and $k=3$ are consequences of estimates for~$I(2)$ and $I(3)$
obtained in the previous section.
\end{proof}

Theorem \ref{th:taug} improves results of Thin \eqref{eq:thin1} and \eqref{eq:thin2}
for all $k$. Cases $k=2$ and $k=3$ were given above.
For $4 \le k \le 12$ by \eqref{eq:M-estimate} and \eqref{eq:I-estimate}
we have
$$
2I(k) \le 0.2896 k + 0.3806 < k < k + 1.
$$
For $13 \le k$ we have
$$
2I(k) \le 2(1 + 13/84(2k-6)) = 52/84k + 1/7 < k < k + 1.$$

It is interesting to check the accuracy of our asymptotic estimate on some
numerical data. Table \ref{t:sum-of-taug} shows values of summatory function
for growing $x$, compared against predictions from \eqref{eq:taug-prediction}. The last
column lists error terms divided by $x^{1/2}$ and fuels our confidence
in the correct order of $O$-bound.

\begin{table}
\center\begin{tabular}{rrrr}
$x$ & $S(x)$
& $\bar S(x)$
& $S(x)-\bar S(x) \over \sqrt x$
\\\hline
10 & 24 & 18 & 1.84 \\
100 & 337 & 324 & 1.31 \\
1000 & 4694 & 4659 & 1.10 \\
10000 & 60857 & 60797 & 0.60 \\
100000 & 750259 & 750002 & 0.81 \\
1000000 & 8920571 & 8920371 & 0.20 \\
10000000 & 103407407 & 103407214 & 0.06 \\
100000000 & 1176104936 & 1176107167 & $-0.22$ \\
1000000000 & 13181419972 & 13181421920 & $-0.06$ \\
\end{tabular}
\par~
\caption{Values of $S(x) = \sum_{0<N\alpha \le x \atop \arg \alpha \in [0,\pi/2)} \taug(\alpha)$
and estimates $\bar S(x)$ by \eqref{eq:taug-prediction}.}
\label{t:sum-of-taug}
\end{table}

\bibliographystyle{ugost2008s}
\bibliography{taue}

\end{document}